\documentclass[10pt]{article}

\usepackage{hyperref}
\usepackage{xcolor}
\usepackage{amssymb}
\usepackage{amsfonts}
\usepackage{amsthm}
\usepackage{amsmath}
\usepackage{float}
\usepackage{bigints}
\usepackage{fullpage}
\usepackage{mathtools}
\usepackage[utf8]{inputenc} 
\usepackage{cancel}
\usepackage{verbatim}

\newtheorem{thm}{Theorem}

\newtheorem{cor}{Corollary}
\newtheorem{prop}{Proposition}
\newtheorem{defn}{Definition}
\newtheorem{rem}{Remark}

\newtheorem{conj}{Conjecture}
\newtheorem{prob}{Problem}
\newtheorem{example}{Example}

\newcommand{\Eu}{\mathbb E}

\newcommand{\B}{\mathbf B}

\renewcommand{\S}{\mathbb{S}}
\newcommand{\HH}{\mathbb{H}}

\newcommand{\K}{\mathbf{K}}
\newcommand{\FF}{\mathbf{F}}
\newcommand{\q}{\mathbf{q}}
\newcommand{\p}{\mathbf{p}}
\newcommand{\U}{\mathbf{U}}
\newcommand{\V}{\mathbf{V}}

\DeclareMathOperator{\inter}{int}
\DeclareMathOperator{\bd}{bd}

\DeclareMathOperator{\card}{card}

\DeclareMathOperator{\relint}{relint}

\newcommand{\ve}{\mathbf{v}}
\newcommand{\PP}{\mathbf{P}}
\newcommand{\HHH}{\mathbf{H}}

\newcommand{\x}{\mathbf{x}}
\newcommand{\y}{\mathbf{y}}
\newcommand{\z}{\mathbf{z}}
\renewcommand{\c}{\mathbf{c}}
\newcommand{\D}{\mathbf{D}}
\newcommand{\h}{\mathbf{h}}
\newcommand{\m}{\mathbf{m}}
\renewcommand{\o}{\mathbf{o}}

\title{From spherical to Euclidean illumination 
\footnote{\hspace*{.35cm} 2010 Mathematics Subject Classification: 52A20, 52A55.}}

\author{K\'{a}roly Bezdek\thanks{Partially supported by a Natural Sciences and 
Engineering Research Council of Canada Discovery Grant.} and Zsolt L\'angi\thanks{Partially supported by the National Research, Development and Innovation Office, NKFI, K-119670, the J\'anos Bolyai Research Scholarship of the Hungarian Academy of Sciences, and grants BME FIKP-V\'IZ of EMMI and \'UNKP-19-4 New National Excellence Program of the Ministry for Innovation and Technology.} 
}

\date{}

\begin{document}

\maketitle

\begin{abstract}
In this note we introduce the problem of illumination of convex bodies in spherical spaces and solve it for a large subfamily of convex bodies. We derive from it
a combinatorial version of the classical illumination problem for convex bodies in Euclidean spaces as well as a solution to that for a large subfamily of convex bodies, which in dimension three leads to special Koebe polyhedra.
\end{abstract}

\section{Introduction}\label{sec:intro}

Let  $\Eu^d$ denote the $d$-dimensional Euclidean space with the unit sphere $\S^{d-1}=\{\x\in\Eu^d\ |\ \langle\x ,\x\rangle=1\}$ centered at the origin $\mathbf{o}$, where $\langle\cdot\rangle$ stands for the standard inner product of $\Eu^d$. We identify $\S^{d-1}$ with the $(d-1)$-dimensional spherical space.  A compact convex set (resp., a compact spherically convex set) with nonempty interior is called a \emph{convex body} in $\Eu^d$ (resp.,  $\S^{d-1}$). (Here, we call a subset of $\S^{d-1}$ {\it spherically convex} if it is contained in an open hemisphere of $\S^{d-1}$ moreover, for any two points of the set the spherical segment, i.e., the shorter unit circle arc connecting them belongs to the set.) Now, recall the following concept of illumination due to Boltyanski \cite{boltyanski1}. (For an equivalent notion of illumination using point sources instead of directions see  Hadwiger \cite{hadwiger2}.) Let $\K$ be a convex body in $\Eu^d$, let $\p \in \bd \K$, i.e., let $\p$ be a boundary point of $\K$, and let $\ve \in \S^{d-1}$ be a direction. We say that $\p$ is \emph{illuminated} from the direction $\ve$, if the half-line with endpoint $\p$ and direction $\ve$ intersects the interior of the convex body $\K$. We say that the directions $\ve_1,\ve_2,\ldots, \ve_m \in \S^{d-1}$ \emph{illuminate} $\K$, if every boundary point is illuminated from some $\ve_i$ for $1 \leq i \leq m$. The smallest number of directions that illuminate $\K$ is called the \emph{illumination number} of $\K$, and is denoted by $I(\K)$. It is easy to see that $I(\K) \geq d+1$ for any convex body $\K$ in $\Eu^d$. On the other hand, since no two distinct vertices of an affine $d$-cube can be illuminated from the same direction, it follows that $I(\K) = 2^d$ holds for any affine $d$-cube $\K$. The following Illumination Conjecture \cite{boltyanski1, hadwiger2} of Hadwiger and Boltyanski is a longstanding open problem in discrete geometry solved only in the plane. For a recent comprehensive survey on the numerous partial results on this conjecture  see \cite{bezdek-khan} (which surveys also the relevant results on the equivalent Covering Conjecture  \cite{gohberg1,  hadwiger1, levi1} as well as Separation Conjecture \cite{bezdek-conj1, bezdek-conj2, soltan1}).


\begin{conj}[Illumination Conjecture]\label{conj:illumination}
The illumination number $I(\mathbf{K})$ of any $d$-dimensional convex body $\mathbf{K}$, $d\geq 3$, is at most $2^d$ and $I(\K) = 2^d$ only if $\mathbf{K}$ is an affine $d$-cube.
\end{conj}

In this paper, we introduce the following notion of illumination (resp., illumination number) for convex bodies in spherical space.

\begin{defn}\label{defn:spherical_ill}
Let $\K \subset \S^d$ be a convex body, and let $\p \in \S^d \setminus \K$. We say that a boundary point $\q \in \bd \K$
is \emph{illuminated from $\p$} if it is not antipodal to $\p$, the spherical segment with endpoints $\p$ and $\q$ does not intersect the interior $\inter \K$ of $\K$, and the greatcircle through $\p$ and $\q$ does. We say that $\K$ is \emph{illuminated} from a set $S \subset \S^d \setminus \K$, if every boundary point of $\K$ is illuminated from at least one point of $S$. The smallest cardinality of a set $S$ that illuminates $\K$ and lies on a $(d-1)$-dimensional greatsphere of $\S^d$ which is disjoint from $\K$,  is called the \emph{illumination number} of $\K$ in $\S^d$, and is denoted by $I_{\S^d}(\K)$.
\end{defn}

We observe that dropping the seemingly artificial restriction that all light sources are contained in a $(d-1)$-dimensional greatsphere of $\S^d$, disjoint from the convex body $\K$, makes the problem of determining $I_{\S^d}(\K)$ trivial. Indeed, choosing any point of $\S^d$, antipodal to an arbitrary interior point of $\K$, illuminates every boundary point of $\K$.

We leave the easy proofs of the following three remarks to the reader.

\begin{rem}\label{rem:central_illumination}
If a set $A\subset  \S^d\setminus \K$ illuminates the convex body $\K$ in $\S^d$ and $A$ is contained in a closed hemisphere $\overline{\HHH}\subset \S^d$ such that $\K \subset \inter \overline{\HHH}$, then $I_{\S^d}(\K) \leq \card (A)$. (Here $\card (\cdot)$ refers to the cardinality of the corresponding set.)
\end{rem}


\begin{rem}\label{reduction-to-Euclidean}
Let $\K^*$ be the {\rm polar body} assigned to the convex body $\K \subset \S^d \subset \Eu^{d+1}$, i.e., let $\K^* := \{ \x \in \S^d : \langle \x, \y \rangle \leq 0 \hbox{ for all } \y \in \K \}$. It is easy to see that $\K^*$ is a convex body in $\S^d$ moreover, $(\K^*)^*=\K$. Clearly, an open hemisphere contains $\K$ if and only if its center is in $- \inter\K^*$.
Note that if some set $S$ of $k$ points in the boundary of such an open hemisphere illuminates $\K$, and $f_{\x}$ is the central projection to the tangent hyperplane of $\S^d$ at the center $\x$ of this hemisphere, then $f_{\x}(S)$ corresponds to a set of $k$ directions which illuminate $f_{\x}(\K)$.
In other words,
\[
I_{\S^d}(\K) = \min \{ I(f_{\x}(\K)) : \x \in  -\inter\K^* \}.
\]
\end{rem}

\begin{rem}\label{2D-spherical-illumination}
Levi \cite{levi1} proved that $I(\K) = 3$ holds for any convex body $\K$ in $\Eu^2$ which is not a parallelogram. Thus, Remark~\ref{reduction-to-Euclidean} implies that $I_{\S^2}(\K)=3$ holds for any convex body $\K$ in $\S^2$.
\end{rem}

In order to state the main illumination results of this paper we need 

\begin{defn}
Let $\K$ be a convex body in $\Eu^d$ (resp., $\S^d$). Recall that a face $F$ of $\K$ is a convex (resp., spherically convex) subset of $\K$ such that for any segment (resp., spherical segment) of $\K$ whose relative interior intersects $F$, the segment (resp., spherical segment) in question is contained in $F$. Then a sequence of faces $F_s \subset \ldots \subset F_{d-1}$ of $\K$ with $\dim F_i = i$ for $i=s,\ldots, d-1$ is called a \emph{partial flag} of $\K$ of length $d-s$. 
\end{defn}

\begin{thm}\label{thm:flag}
Let $\K$ be a convex body in $\S^d$, $d > 2$, whose boundary contains a partial flag of length $d-2$. Then $I_{\S^d}(\K)=d+1$. In particular, for any convex polytope $\PP$ in $\S^d$, $d>2$, we have $I_{\S^d}(\PP)=d+1$.
\end{thm}

\begin{cor}\label{spherical-approx}
For any convex body $\K$ in $\S^d$, $d>2$ and for any $\varepsilon > 0$, there is a convex body $\K' \subset \K$ such that $\K \setminus \K'$ is contained in a spherical cap of radius $\varepsilon$, and $I_{\S^d}(\K') = d+1$.
\end{cor}

Although it is easy to see that for any smooth convex body $\K$ in $\S^d$, $d > 2$ we have $I_{\S^d}(\K)=d+1$, the question of finding a proper extension of Theorem~\ref{thm:flag} to all convex bodies seems to raise an open problem. 

\begin{prob}\label{new-fundamental}
Prove or disprove that $I_{\S^d}(\K)=d+1$ holds for any convex body $\K$ in $\S^d, d>2$.
\end{prob}

Next, we apply Theorem~\ref{thm:flag} to illumination numbers of convex bodies in $\Eu^d$. Motivated by the notion of combinatorial equivalence for convex polytopes, we introduce the following notion. Note that, restricted to convex polytopes, this notion is equivalent to the usual concept of combinatorial equivalence.

\begin{defn}\label{defn:combequiv}
Let $\K, \K' \subset \Eu^d$ be convex bodies. If there is a homeomorphism $h : \bd \K \to \bd \K'$ such that for any $X \subset \bd \K$, $X$ is a face of $\K$ if and only if $h(X)$ is a face of $\K'$, then we say that $\K$ and $\K'$ are {\rm combinatorially equivalent}.
\end{defn}

Note that any two strictly convex bodies of $ \Eu^d$ are combinatorially equivalent. Furthermore, combinatorial equivalence is an equivalence relation on the family of convex bodies, the equivalences classes of which we call \emph{combinatorial classes}.

\begin{example}
Let $0 < \alpha < \pi$, and for any positive integer $k$, let $\p_k = \left( \cos \frac{2\alpha}{3^k}, \sin \frac{2\alpha}{3^k} \right) \in \Eu^2$, and $\q_k = \left( \cos \frac{\alpha}{3^k} , \sin \frac{\alpha}{3^k} \right) \in \Eu^2$. Furthermore, for any value of $k$, define $\HHH_k$ as the closed half plane in $\Eu^2$ containing $\o$ in its interior and $\p_k$ and $\q_k$ on its boundary, and let $\HHH_{-k}$ be the reflected copy of $\HHH_k$ about $\o$. Finally, set
\[
\K = \B^2 \cap \bigcap_{k=1}^\infty \HHH_k
\]
and
\[
\K' = \B^2 \cap \bigcap_{k \in \mathbb{Z} \setminus \{ 0 \}} \HHH_k,
\]
where $\B^2$ is the closed Euclidean unit disk centered at $\o$. Then, clearly, there is a \emph{bijection} $h : \bd \K \to \bd \K'$ such that $X \subset \bd \K$ is a face of $\K$ if and only if $h(X)$ is a face of $\K'$, but there is no \emph{homeomorphism} with the same property.
\end{example}

\begin{defn}
For any convex body $\K$ in $\Eu^d$, the smallest number $k$ such that some element of the combinatorial class of $\K$ can be illuminated from $k$ directions is called the \emph{combinatorial illumination number of $\K$}, and is denoted by $I_c(\K)$.
\end{defn}

\begin{thm}\label{cor:combinatorial_illumination}
For any convex body $\K$ in $\Eu^d$, $d>2$, whose boundary contains a partial flag of length $d-2$, we have $I_c(\K) = d+1$. In particular, for any convex polytope $\PP$ in $\Eu^d$, $d>2$, we have $I_c(\PP) = d+1$.
\end{thm}

In $\Eu^3$, we can prove more. In order to state our result, recall that the combinatorial class of every convex polyhedron $\PP\subset\Eu^3$ contains special convex polyhedra, called \emph{Koebe polyhedra}, which are combinatorially equivalent to $\PP$ and are \emph{midscribed} to $\S^2$, i.e., their edges are tangent to $\S^2$.

\begin{thm}\label{thm:Koebe}
If $\PP$ is a convex polyhedron in $\Eu^3$, then the combinatorial class of $\PP$ contains a Koebe polyhedron $\PP'$ with $I(\PP')=4$.
\end{thm}

We close this section with the following polar description of $I_{\S^d}(\K)$, which is the spherical analogue of the Separation Lemma in \cite{bezdek-conj1}. For the sake of completeness the Appendix of this paper contains a proof of Theorem~\ref{dual-description}. Here we recall that an exposed face of the convex body $\K$ in $\S^d$ (resp., $\Eu^d$) is the intersection of a supporting $(d-1)$-dimensional greatsphere (resp., supporting hyperplane) of $\K$ with $\K$.

\begin{thm}\label{dual-description}
$I_{\S^d}(\K)$ is equal to the minimum number of open hemispheres of $\S^d$ whose boundaries all pass through a common point in the interior of the polar convex body $\K^*$ and have the property that every exposed face of $\K^*$ is contained in at least one of the open hemispheres.
\end{thm}

Thus, Theorem~\ref{dual-description} implies in a straightforward way that Problem~\ref{new-fundamental} is equivalent to the following spherical question (resp., Euclidean question obtained from it via central projection), both of which one can regard as a natural counterpart of the Separation Conjecture \cite{bezdek-conj1, bezdek-conj2, soltan1}.

\begin{prob}\label{dual-new-fundamental}
Prove or disprove that if $\K'$ is an arbitrary convex body in $\S^d$ (resp., $\Eu^d$), $d>2$, then there exist $\x\in\inter\K'$ and $d+1$ open hemispheres (resp., open halfspaces) of $\S^d$ (resp., $\Eu^d$) whose boundaries contain $\x$ such that every exposed face of $\K'$ is contained in at least one of the open hemispheres (resp., open halfspaces).
\end{prob} 

On the one hand, if $\K'$ is a strictly convex body in $\S^d$ (resp., $\Eu^d$), $d>2$, then there is an easy positive answer to Problem~\ref{dual-new-fundamental}. On the other hand, using Theorem~\ref{dual-description} one can derive from Theorem~\ref{thm:flag} 

\begin{cor}\label{dual-new-fundamental-for-polytopes}
If $\PP$ is an arbitrary convex polytope in $\S^d$ (resp., $\Eu^d$), $d>2$, then there exist $\x\in\inter\PP$ and $d+1$ open hemispheres (resp., open halfspaces) of $\S^d$ (resp., $\Eu^d$) whose boundaries contain $\x$ such that every (exposed) face of $\PP$ is contained in at least one of the open hemispheres (resp., open halfspaces).
\end{cor}

In the rest of the paper we prove the theorems stated.

\section{Proofs of Theorem~\ref{thm:flag} and Corollary~\ref{spherical-approx}}

First, observe that no convex body in $\S^d$ can be illuminated from fewer than $d+1$ points (lying on a $(d-1)$-dimensional greatsphere, which is disjoint from the convex body). (This statement follows from the analogue Euclidean result via central projection between $\S^d$ and its corresponding tangent hyperplane in  $\Eu^{d+1}$.) Thus, we need to find a $(d+1)$-element set, contained in a $(d-1)$-dimensional greatsphere not intersecting $\K$, that illuminates $\K$ in $\S^d$.

We prove Theorem~\ref{thm:flag} by induction on $d$ for all $d>2$. So, let us start by assuming that Theorem~\ref{thm:flag} holds for convex bodies with partial flags of length $d-3$ in $\S^{d-1}$, and let $\K$ be a convex body in $\S^d$ with a partial flag $F_{2} \subset \ldots \subset F_{d-1} =  F$ on its boundary. (If $d=3$, i.e., if $\K$ is a convex body with partial flag $F_2$ in $\S^3$, then $F_2$ is a convex body in $\S^2$ and therefore Remark~\ref{2D-spherical-illumination} implies $I_{\S^2}(F_2)=3$, i.e., it guarantees the inductive assumption for this case.) For simplicity, we refer to the $(d-1)$-dimensional greatsphere $H$ of $\S^d$ containing $F$ as the \emph{equator}, and the open hemisphere $\HHH$ bounded by $H$ and containing $\inter \K$ as the \emph{northern hemisphere}.
Furthermore, for any open neighborhood $\mathbf{U}$ of a point $\x \in H$, we call $\mathbf{U} \cap \HHH$ and $\mathbf{U} \cap (-\HHH)$ the \emph{northern} and \emph{southern} halves of $\mathbf{U}$.

Let $\p$ be a relative interior point of $F$. Note that since $\p \in \mathrm{relint}\  F$, and $\dim F= d-1$,
the northern half of a sufficently small neighborhood of $\p$ is contained in $\inter \K$. 
We show that the point $-\p$, antipodal to $\p$, illuminates every point of $\bd \K \setminus F$.
Indeed, $\HHH$ can be decomposed into semicircles starting at $-\p$ and ending at $\p$. 
As each such semicircle intersects the northern half of every neighborhood of $\p$, each of these semicircles contains interior points of $\K$.
Thus, every point of $\bd \K \setminus H = \bd \K \setminus F$ is illuminated from $-\p$.

Note that if some $\x \in \bd \K$ is illuminated from a point $\y$, then $\y$ has a neighborhood $\V$ such that $\x$ is illuminated from any point of $\V$.
Thus, if $D$ is an arbitrary compact subset of $\bd \K \setminus F$, then $-\p$ has a neighborhood $\V$ such that every point of $D$ is illuminated from every point of $\V$.

Since $F$ is a convex body with a partial flag of length $d-3$ in the $(d-1)$-dimensional spherical space $H$, there is a set $S'$ of $d$ points, contained in a $(d-2)$-dimensional greatsphere $G$ of $H$, which illuminate $F$ in $H$, i.e., for every relative boundary point $\q \in \mathrm{relbd}\  F$ there is a point $\x \in S'$ such that the semicircle starting at $\x$ and passing through $\q$ intersects $\mathrm{relint}\  F$.
Since $\K$ is spherically convex and $\dim F= d-1$, this also implies that if $\x'$ is in the southern half of a suitable neighborhood of $\x$, then $\x'$ illuminates $\K$ at $\q$ in $\S^d$, i.e., the semicircle starting at $\x'$ and passing through $\q$ intersects $\inter \K$.
Observe that (since $\dim F = d-1$) $\K$ is illuminated at every relative interior point of $F$ from any point in the southern hemisphere $-\HHH$.
Thus, there is a family of sets $\U_1, \U_2, \ldots, \U_d$, each being the southern half of a suitable neighborhood of a point of $S'$, such that any $d$-tuple
$\x_i \in \U_i$, $i=1,2,\ldots, d$, illuminates $\K$ at every point of $F$ in $\S^d$. By compactness arguments, there is a set ${L} \subset \bd \K$ containing $F$ and open in $\bd \K$ that has the same property as $F$, i.e., there are some suitable sets $\U'_i$, $i=1,2,\ldots,d$, each being the southern half of a suitable neighborhood of a point of $S'$ such that any $d$-tuple $\x_i \in \U'_i$, $i=1,2,\ldots, d$, illuminates $\K$ at every point of $L$ in $\S^d$.

Note that $\bd \K \setminus {L}$ is compact. Hence, we may choose a point $\x_0 \in \HHH$ and sufficiently close to $-\p$ such that $\x_0$ illuminates  $\K$ at every point of
$\bd \K \setminus {L}$. Let $H'$ be the $(d-1)$-dimensional greatsphere spanned by $G$ and $\x_0$. Note that by our choice of $G$, chosen as a $(d-2)$-dimensional greatsphere $G$ of $H$ with $S' \subset G$, $G$ does not intersect $F$, which implies that $G$ strictly separates $F$ and $-\p$ in $H$. Since $\x_0$ is sufficiently close to $-\p$, from this it follows that $H'$ does not intersect $\K$. Since $H'$ is a rotated copy of $H$ around $G$, it intersects the southern half of any neighborhood of any point of $G$. Thus, $H'$ intersects $\U'_i$ for all values of $i$. Pick some point $\x_i$ from $\U'_i \cap H'$ for $i=1,2,\ldots,d$, and set $S= \{ \x_0, \ldots, \x_d \} \subset H'$. Since $S \setminus \{ \x_0 \}$ illuminates $\K$ at every point of ${L}$ by the choice of the $\U'_i$s, we have constructed a set of $(d+1)$ points, contained in the boundary of an open hemisphere containing $\K$, that illuminates $\K$. Thus, $I_{\S^d}(\K)=d+1$, finishing the proof of Theorem~\ref{thm:flag}.

Now, we prove Corollary~\ref{spherical-approx}. Let $\p$ be an exposed point of $\bd \K$, i.e., a boundary point of $\K$ that can be obtained as an intersection of $\K$ with a supporting $(d-1)$-dimensional greatsphere of $\K$ in $\S^d$. (The existence of an exposed point is well known see for example, Theorem 1.4.7 in \cite{Sc}.) Then we can truncate $\K$ near $\p$ by a $(d-1)$-dimensional greatshere such that the closure of the part removed is contained in an open spherical cap of radius $\varepsilon$. Continuing the truncation process by subsequent greatspheres, we can construct a truncated convex body $\K'$ whose boundary contains a partial flag of length $d-2$, and has the property that $\K \setminus \K'$ is covered by a spherical cap of radius $\varepsilon$. By Theorem~\ref{thm:flag}, $I_{\S^d}(\K') = d+1$.

\section{Proof of Theorem~\ref{cor:combinatorial_illumination}}

Since for any convex body $\K$, we have $I(\K) \geq d+1$, therefore $I_c(\K) \geq d+1$. We show that $I_c(\K) \leq d+1$.

Imagine $\Eu^d$ as a tangent hyperplane of the sphere $\S^d$, embedded in $\Eu^{d+1}$ in the usual way. Let $h : \Eu^d \to \S^d$ be the central projection of $\Eu^d$ to $\S^d$. Then ${\K}':=h(\K)$ is a spherical convex body, having a partial flag of length $d-2$. By Theorem~\ref{thm:flag}, there is a greatsphere $H$ of $\S^d$ disjoint from ${\K}'$, and a $(d+1)$-element point set $S' \subset H$ such that $S'$ illuminates ${\K}'$.
Let $\HHH$ be the open hemisphere bounded by $H$ that contains ${\K}'$, and let $\c$ be the spherical center of $\HHH$.
Let $h_{\c} : \HHH \to T_{\c} \S^d$ be the central projection of $\HHH$ to the tangent hyperplane of $\S^d$ at $\c$.
Since $T_{\c} \S^d$ is a $d$-dimensional Euclidean space, ${\K}'':=h_{\c}({\K}')$ is a $d$-dimensional Euclidean convex body.
$h_{\c} \circ h |_{\bd \K}$ is a homeomorphism, and $h_{\c} \circ h$ maps faces of $\K$ to faces of ${\K}''$.
Thus, $\K$ and ${\K}''$ are combinatorially equivalent.
Furthermore, the images of the great circle arcs starting at a point $\q$ of $H$ are parallel lines in $T_{\c} \S^d$ starting at the `ideal point' $h_{\c}(\q)$ of
$T_{\c} \S^d$.
Hence, the set of ideal points of $T_{\c} \S^d$ corresponding to $S'$ is a set of $d+1$ directions that illuminates ${\K}''$.

\section{Proof of Theorem~\ref{thm:Koebe}}

To prove the theorem, we adopt some ideas from the proof of Theorem 3 in \cite{Koebe}. 
Let $\PP$ be a Koebe polyhedron, i.e., a convex polyhedron in $\Eu^3$ whose edges are tangent to $\S^2$. Then there are two families of circles on $\S^2$ associated to $\PP$ \cite{Brightwell, Koebe}.
The elements of the first family, called \emph{face circles} are the incircles of the faces of $\PP$; each such circle touches the edges of a face of $\PP$ at the points where the edges touch $\S^2$. We denote these circles by $f_j$, $j=1,2,\ldots, m$, where $m$ is the number of faces of $\PP$.
The elements of the second family are called \emph{vertex circles}. These circles, denoted by $v_i$, $i=1,2,\ldots,n$, where $n$ is the number of vertices of $\PP$, are circles on $\S^2$ that contain the tangency points on all the edges of $\PP$ starting at a given vertex of $\PP$.
The tangency graphs of these two families are dual graphs.
If $T : \S^2 \to \S^2$ is a M\"obius transformation then $T$ maps the two circle families of $\PP$ into two circle families which are associated to another Koebe polyhedron, which, with a little abuse of notation, we denote by $T(\PP)$. It is known \cite{Brightwell} that if $\PP$ and $\PP'$ are two combinatorially equivalent Koebe polyhedra, then there is a M\"obius transformation $T : \S^2 \to \S^2$ satisfying $T(\PP)=\PP'$.

In our proof, we regard $\S^2$ as the set of the `points at infinity' of the Poincar\'e ball model of the hyperbolic space $\HH^3$, which is identified with the interior of the Euclidean unit ball bounded by $\S^2$. Then every face circle $f_j$ of $\PP$ is the set of the ideal points of a unique hyperbolic plane $F_j$, and the same holds for every vertex circle of $\PP$; we denote the hyperbolic plane associated to the vertex circle $v_i$ by $V_i$. Furthermore, we denote by $\FF_j$ the closed hyperbolic halfspace bounded by $F_j$ which is disjoint from any hyperbolic plane associated to any other face circle of $\PP$, and define $\V_i$ similarly for any vertex circle $v_i$ of $\PP$. It is worth noting that every M\"obius transformation of $\S^2$ corresponds to a hyperbolic isometry in the Poincar\'e ball model, and vice versa.

Let $\D := \HH^3 \setminus \left( \bigcup_{i=1}^n \V_i \cup \bigcup_{j=1}^m \FF_j \right)$, and note that for every value of $j$, $F_j \cap \bd \D$ is a closed hyperbolic polygon $P_j$ with ideal vertices and nonempty relative interior in $\HH^3$. Consider some point $\p$ in $\relint P_1$. Let $h_{\p} : \HH^3 \to \HH^3$ be a hyperbolic isometry that maps $\p$ into $\o$, and let $T_{\p}$ be the corresponding M\"obius transformation.
Then the first face $F$ of $T_{\p}(\PP)$ contains the origin $\o$. Let $\m$ be the outer unit normal vector of the Euclidean plane through $F$.
Then the angle between $\m$ and the outer unit normal vector of any other face of $T(\PP)$ is obtuse, which implies that the projection of $T(\PP) \setminus F$
onto the Euclidean plane through $F$ (or in other words, the projection in the direction of $\m$) is $\inter F$. In other words,
$\m$ illuminates every point of $\bd (T(\PP) ) \setminus F$. Note that if there are three directions $\m_1, \m_2, \m_3$ which illuminate $F$ in the plane containing it, then the vectors $\m_i - \varepsilon \m$, where $i=1,2,3$, illuminate every point of $F$ in $\Eu^3$. On the other hand, by a result of Levi \cite{levi1}, apart from parallelograms, the illumination number of every plane convex body is $3$.

Thus, to prove the assertion we need to show that, using a suitable point $\p$ in $F_1$, the first face $F$ of $T_{\p}(\PP)$ is not a parallelogram.
Assume that $F$ is a parallelogram. Then clearly, the first face of $\PP$ is a quadrangle, and thus, there are four tangency points on the first face circle $f_1$ of $\PP$. Let these points be $\q_1,\q_2,\q_3,\q_4$ in cyclic order. Let the unique hyperbolic line with ideal points $\q_1,\q_3$ be denoted by $L_1$, and the line with ideal points $\q_2,\q_4$ be denoted by $L_2$. These lines are contained in $F_1$.

Since $F$ is a parallelogram and is circumscribed about a circle, $F$ is a rhombus, and the tangent points $T(\q_1),T(\q_2), T(\q_3)$ and $T(\q_4)$ are the vertices of a rectangle. Thus, the midpoint of both Euclidean open segments $(T(\q_1),T(\q_3))$ and $(T(\q_2),T(\q_4))$ is the origin $\o$.
Note that these segments represent the hyperbolic lines $h_{\p}(L_1)$ and $h_{\p}(L_2)$. Thus, in this case $\o$ is the intersection point of $h_{\p}(L_1)$ and $h_{\p}(L_2)$. Since hyperbolic isometries preserve incidences, it follows that $\p$ is the intersection point of $L_1$ and $L_2$. On the other hand, as $P_1$ contains infinitely many relative interior points, we may choose some point in $\relint P_1$ different from this point, implying that in this case the first face $F$ of $T_{\p}(\PP)$ is not a parallelogram.

\section{Appendix: Proof of Theorem~\ref{dual-description}}

The following proof is a spherical analogue of the proof of the Separation Lemma in \cite{bezdek-conj1}.

\begin{defn}
Let $\K \subset \S^d \subset \Eu^{d+1}$ be a convex body and $F$ be an exposed face of $\K$. We define the conjugate face of F as a subset of the polar convex body ${\K}^*= \{ \x \in \S^d : \langle \x, \y \rangle \leq 0 \hbox{ for all } \y \in \K \}$ given by
\begin{equation}\label{conjugate}
\hat{F}:=\{\x\in{\K}^*\ |\ \langle\x ,\y\rangle=0\ {\text for\  all}\ \y\in F\}.
\end{equation}
\end{defn}
\noindent One should keep in mind that $\hat{F}$ depends also on $\K$ and not only on $F$. So, if we write $\hat{\hat{F}}$, then it means
$(\hat{F})^{\hat{}}$, where the right-hand circumflex refers to the spherical polar body ${\K}^*$. If $\x\in\S^d$, then let $\HHH_{\x}$ denote the open hemisphere of $\S^d$ with center $\x$.

The following statement (which one can regard as a natural spherical analogue of Theorem 2.1.4 in \cite{Sc}) shows that exposed faces behave well under polarity in spherical spaces. 

\begin{prop}\label{exposed-face-representation}
Let $\K \subset \S^d \subset \Eu^{d+1}$ be a convex body and $F$ be an exposed face of $\K$. Then $\hat{F}$ is an exposed face of ${\K}^*$ with $\hat{F}=\bigcap_{\y\in F}\left(\bd{\HHH_{\y}}\cap{\K}^*\right)$, where $\S^d\setminus\HHH_{\y}$ is a closed supporting hemisphere of ${\K}^*$ for every $\y\in F$. Moreover,  $\hat{\hat{F}}=F$ and so, $F\mapsto \hat{F}$ is a bijection between the exposed faces of $\K$ and ${\K}^*$.
\end{prop}
\begin{proof} Without loss of generality we may assume that $F$ is a proper exposed face of $\K$, i.e., there exists an open hemisphere $\HHH_{\x_0}$ of $\S^d$ such that $\emptyset\neq F=\K\cap\bd\HHH_{\x_o}$ and $\K\cap\HHH_{\x_0}=\emptyset$. It follows that $\x_0\in\hat{F}$ and so, $\hat{F}\neq\emptyset$. Now, if $\y\in F$, then $\K^*\cap\HHH_{\y}=\emptyset$ and $\x\in\bd \HHH_{\y}$ holds for all $\x\in\hat{F}$. Thus, $\hat{F}\subseteq\bigcap_{\y\in F}(\K^*\cap\bd \HHH_{\y})$, where $\K^*\cap\HHH_{\y}=\emptyset$ holds for all $\y\in F$. On the other hand, if $\z\in \bigcap_{\y\in F}(\K^*\cap\bd \HHH_{\y})$ with $\K^*\cap\HHH_{\y}=\emptyset$ for all $\y\in F$, then $\z\in\K^*$ with $\langle\z,\y\rangle=0$ for all $\y\in F$ (and therefore $\z\in\hat{F}$) implying that $\bigcap_{\y\in F}(\K^*\cap\bd \HHH_{\y})\subseteq\hat{F}$. Thus, $\hat{F}=\bigcap_{\y\in F}(\K^*\cap\bd \HHH_{\y})$ with $\K^*\cap\HHH_{\y}=\emptyset$ for all $\y\in F$. As $\K^*\cap\bd \HHH_{\y}$ is a proper exposed face of $\K^*$ for all $\y\in F$ therefore $\hat{F}$ is a proper exposed face of $\K^*$.  Applying the above argument to the exposed face $\hat{F}$ of $\K^*$ one obtains in a straightforward way that $\hat{\hat{F}}=F$. This completes the proof of Proposition~\ref{exposed-face-representation}.
\end{proof}

\begin{prop}\label{separation-by-hemisphere}
Let $\K \subset \S^d \subset \Eu^{d+1}$ be a convex body and $H$ be a $(d-1)$-dimensional greatsphere of $\S^d$ with $H\cap\K=\emptyset$. Then
$\q \in \bd \K$ is illuminated from $\p\in H$ with $\p\neq-\q$ if and only if $\hat{F}\subset\HHH_{\p}$, where $F$ denotes the exposed face of $\K$ having smallest dimension and containing $\q\in\bd \K$.
\end{prop}
\begin{proof} Let $\HHH_{\h}$ be  the open hemisphere with center $\h$ and boundary $H$ in $\S^d$ such that $\K\subset\HHH_{\h}$.   
Clearly, $-\h\in\bd{\HHH_{\p}}\cap\inter{\K^*} $. Proposition~\ref{exposed-face-representation} implies that $F=\bigcap_{\x\in \hat{F}}\left(\bd{\HHH_{\x}}\cap{\K}^*\right)$, where $\HHH_{\x}\cap\K=\emptyset$ for all $\x\in\hat{F}$. Let $[\p,\q)$ denote the spherical segment of $\S^d$ with endpoints $\p$ and $\q$ containing $\p$ and not containing $\q$. Now, $\q \in \bd \K$ is illuminated from $\p\in H$ if and only if 
$[\p,\q)\subset\bigcap_{\x\in\hat{F}}\HHH_{\x}$, which is equivalent to $\p\in\bigcap_{\x\in\hat{F}}\HHH_{\x}$ (because $\pm\q\in \bigcap_{\x\in\hat{F}}\bd \HHH_{\x}$). Finally,  $\p\in\bigcap_{\x\in\hat{F}}\HHH_{\x}$ holds if and only if $\hat{F}\subset\HHH_{\p}$. This finishes the proof of Proposition~\ref{separation-by-hemisphere}.
\end{proof}

Now, the following statement follows from Proposition~\ref{separation-by-hemisphere} and its proof in a straightforward way.

\begin{cor}\label{separation-by-hemispheres}
Let $\K \subset \S^d \subset \Eu^{d+1}$ be a convex body and $H$ be a $(d-1)$-dimensional greatsphere of $\S^d$ with $H\cap\K=\emptyset$. Let $\HHH_{\h}$ be  the open hemisphere with center $\h$ and boundary $H$ in $\S^d$ such that $\K\subset\HHH_{\h}$. Then the point set $\{\p_1,\dots ,\p_n\}\subset H$ illuminates $\K$ if and only if the open hemispheres $\HHH_{\p_1},\dots,\HHH_{\p_n}$ with $-\h\in\bd{\HHH_{\p_i}}\cap\inter{\K^*} $ for $1\leq i\leq n$ have the property that every (proper) exposed face of $\K^*$ is contained in at least one of the open hemispheres.
\end{cor}

Finally, Corollary~\ref{separation-by-hemispheres} yields Theorem~\ref{dual-description}.

\bigskip 

{\bf Acknowledgements}

We are indebted to the anonymous referee for careful reading and valuable comments.

\bigskip


\noindent K\'aroly Bezdek \\
\small{Department of Mathematics and Statistics, University of Calgary, Canada}\\
\small{Department of Mathematics, University of Pannonia, Veszpr\'em, Hungary\\
\small{E-mail: \texttt{bezdek@math.ucalgary.ca}}

\bigskip

\noindent and

\bigskip

\noindent Zsolt L\'angi \\
\small{MTA-BME Morphodynamics Research Group and Department of Geometry}\\ 
\small{Budapest University of Technology and Economics, Budapest, Hungary}\\
\small{\texttt{zlangi@math.bme.hu}}

\end{document}